\documentclass[a4paper, 11pt, reqno]{amsart}
\usepackage{amsmath,amsfonts,amssymb,amsthm,enumerate}
\usepackage[hidelinks]{hyperref}
\usepackage[utf8]{inputenc}

\newtheorem{thm}{Theorem}[section]
\newtheorem{cor}[thm]{Corollary}
\newtheorem{prop}[thm]{Proposition}
\newtheorem{lem}[thm]{Lemma}
\newtheorem*{op}{Open Problem}

\theoremstyle{definition}
\newtheorem{defn}[thm]{Definition}
\newtheorem{exas}[thm]{Examples}
\newtheorem{rem}[thm]{Remark}

\newcommand{\defeq}{\mathrel{\mathop:}=}
\newcommand{\eqdef}{=\mathrel{\mathop:}}
\newcommand{\B}{\mathbb B}
\newcommand{\Z}{\mathbb Z}
\newcommand{\mcG}{\mathcal G}
\newcommand{\mcO}{\mathcal O}
\newcommand{\mcT}{\mathcal T}
\newcommand{\mcGo}{\mathcal G^{(0)}}

\let\on\operatorname
\let\epsilon\varepsilon
\let\theta\vartheta
\let\phi\varphi

\begin{document}

\title{Congruence-simplicity of Steinberg algebras of non-Hausdorff ample groupoids over semifields}

\author{Tran Giang Nam}
\address{Institute of Mathematics, VAST \\ 18 Hoang Quoc Viet, Cau Giay, Hanoi, Vietnam}
\email{tgnam@math.ac.vn}

\author{Jens Zumbr\"agel}
\address{Faculty of Computer Science and Mathematics \\ University of Passau, Germany}
\email{jens.zumbraegel@uni-passau.de}


\begin{abstract} We investigate the algebra of an ample groupoid, introduced by Steinberg, over a semifield~$S$.  In particular, we obtain a complete characterization of congruence-simpleness for Steinberg algebras of second-countable ample groupoids, extending the well-known characterizations when~$S$ is a field.  We apply our congruence-simplicity results to tight groupoids of inverse semigroup representations associated to self-similar graphs. \medskip

\end{abstract}

\maketitle

\section{Introduction}

Convolutional algebras over topological groupoids have been a subject of intensive study, \textit{e.g.}, when describing C$^*$-algebras~\cite{r:agatca,p:gisatoa}, inverse semigroup algebras~\cite{s:agatdisa} and Leavitt path algebras~\cite{cfst:aggolpa}.  In particular, so-called Steinberg algebras~\cite{s:agatdisa} are defined over ample groupoids, \textit{i.e.}, étale groupoids with totally disconnected unit space.
Within these investigations, questions on the simplicity of the algebras received quite some attention recently, see, \textit{e.g.}, \cite{bcfs:soaateg,s:spasoegawatisa,cepss:soaatnhg,ss:soisaega}.

While the unit space of the underlying groupoid is usually assumed to be locally compact Hausdorff, the groupoid itself is frequently non-Hausdorff in interesting examples, such as the groupoid of germs of self-similar group actions~\cite{n:goegasa} or of self-similar graphs~\cite{ep:ssgautokanca,cepss:soaatnhg}.
A surprising new phenomenon that occurs in the non-Hausdorff scenario is that simpleness of the Steinberg algebra may actually depend on the ground field.  In fact, it has been observed that the algebra over the groupoid of germs associated to the Grigorchuk group is simple over fields of characteristic zero, but not over the base field $\mathbb F_2$~\cite{n:goegasa,cepss:soaatnhg}.
Lying behind these observations are certain nonzero ideals of “singular functions”, which can only exist in the non-Hausdorff case.

Semirings and semifields have emerged in the literature as natural generalizations of rings and fields in a “non-additive” setting, and found applications in diverse areas such as computer science, theoretical physics and cryptography (see, \textit{e.g.}, the monograph~\cite{g:agttlosataimais} for an overview).
This paper continues the investigation of classes of congruence-simple infinite semirings by Katsov and the authors~\cite{knz:ososacs,knz:solpawcias,nz:osaohagocs}.  In particular, Katsov and the authors have investigated Leavitt path algebras over semirings~\cite{knz:solpawcias}, while the present authors initiated a study of Steinberg algebras of Hausdorff ample groupoids over semirings~\cite{nz:osaohagocs}.
This semiring setup showcases interesting novel attributes of the Steinberg algebras; for example, contrary to the ring case, it turns out that the algebra of a finite inverse semigroup over a semiring is not necessarily isomorphic to its associated Steinberg algebra (see \cite[Prop.~2.12]{nz:osaohagocs}).

As the realm of non-Hausdorff groupoids yields some unexpected behavior and interesting challenges, in this work we suggest an investigation of Steinberg algebras over general ample groupoids.
We provide a classification of congruence-simpleness of such Steinberg algebras in case the underlying groupoid is second-countable.  In doing so, we give a universal property of Steinberg algebras over the Boolean semifield~$\B$ and introduce congruence versions of various ideals of singular functions as introduced by Nekrashevych~\cite{n:goegasa} and by Clark et al~\cite{cepss:soaatnhg}.
We also examine examples given by groupoids of self-similar graphs and related classes, with respect to simplicity of their Steinberg algebras.
Here we should mention that congruence-simple semirings have found applications in constructions of novel semigroup actions for a potential use in public-key cryptosystems (see, \textit{e.g.}, \cite{mmr:pkcbosa}).
In this regard, a fundamental problem is thus to classify congruence-simple semirings, in particular additively idempotent congruence-simple semirings.

The article is organized as follows.  In the subsequent section we introduce the Steinberg algebra over a commutative semiring (Definition~\ref{def:steinberg}) and give some preparatory results; in particular, we establish a universal property of Steinberg algebras over the Boolean semifield~$\B$ (Corollary~\ref{univproperty}).
Section~\ref{sec:simpleness} contains a study of congruence-simple Steinberg algebras, in which we introduce the aforementioned congruences (Lemma~\ref{lem:cong1} and Lemma~\ref{lem:cong2}) and present our main result (Theorem~\ref{thm:simpleness}).
Then in Section~\ref{sec:examples} we apply our results to tight groupoids of inverse semigroup representations and discuss examples coming from self-similar graphs and self-similar group actions.

\section{Steinberg algebras over commutative semirings}

In this section we introduce the Steinberg algebra of an ample groupoid over a commutative semiring.  Originally, the construction of the algebra was given for ample groupoids over commutative rings~\cite{s:agatdisa}.  In our previous work~\cite{nz:osaohagocs} we studied Steinberg algebras of Hausdorff ample groupoids over commutative semirings, while here we do not assume the groupoid to be Hausdorff.

We first briefly recall some basic notions of semirings, semimodules and algebras; for more background information we refer to~\cite{g:sata}.

By a \emph{hemiring} $(R, +, \cdot)$ we mean an algebraic structure such that its additive reduct $(R, +)$ is a commutative monoid with neutral element~$0$, its multiplicative reduct $(R, \cdot)$ is a semiring, and the distributive laws $r (s + t) = r s + r t$, $(r + s) t = r t + s t$ hold for $r, s, t \in R$.
When there is also a neutral element~$1$ with respect to multiplication we refer to a \emph{semiring}; one always has $1 \ne 0$ except in the \emph{trivial} semiring $R = \{ 0 \}$.
If the additive reduct of a hemiring is actually an abelian group then the hemiring is a \emph{ring}, otherwise we speak of a \emph{proper} hemiring.  The hemiring is said to be \emph{commutative} if its multiplication is commutative.
A nontrivial commutative semiring is called a \emph{semifield} if every nonzero element has a multiplicative inverse.  An important example of a proper semifield is the \emph{Boolean semifield} $\B \defeq (\{ 0, 1 \}, \on{or}, \on{and})$.

A \emph{congruence} on a hemiring $(R, +, \cdot)$ is an equivalence relation~$\sim$ on~$R$ such that $r \sim s$ implies $t + r \sim t + s$, $t r \sim t s$, $r t \sim s t$ for all $r, s, t \in R$.  If there exist only the trivial congruences, namely the full one $R \times R$ and the diagonal one $\triangle_R \defeq \{ (r, r) \mid r \in R \}$, then the hemiring is called \emph{congruence-simple}.
As usual, a map $\phi \colon R \to S$ between hemirings~$R$ and~$S$ is called \emph{homomorphism} if $\phi(a + b) = \phi(a) + \phi(b)$, $\phi(a b) = \phi(a) \phi(b)$, $\phi(0) = 0$ holds for all $a, b \in R$; and an \emph{isomorphism} is a bijective homomorphism.
It is not hard to see that a hemiring~$R$ is congruence-simple if and only if every nonzero homomorphism $R \to S$ into a hemiring~$S$ is injective, cf.~\cite[Rem.~2.1]{nz:osaohagocs}.

Let~$S$ be a semiring.  By a (left) \emph{semimodule} $(M, +)$ over~$S$ we mean a commutative monoid with neutral element~$0$, together with an $S$-multiplication $S \times M \to M$, $(s, m) \mapsto s m$ such that $(r + s) m = r m + s m$, $r (m + n) = r m + r n$, $r (s m) = (r s) m$, $r 0 = 0$, $0 m = 0$, $1 m = m$ for all $r, s \in S$, $m, n \in M$.  Homomorphisms and isomorphisms between $S$-semimodules are defined in a standard manner.  An $S$-semimodule is called \emph{free} if it is isomorphic to a direct sum of the canonical $S$-semimodule~$S$.

Now let~$S$ be a commutative semiring.  An \emph{algebra} $(A, +, \cdot)$ over~$S$ is an $S$-semimodule which is also a hemiring such that $(s a) b = s (a b) = a (s b)$ holds for all $s \in S$, $a, b \in A$.
For example, if~$G$ is a group, then the \emph{group algebra}~$S G$ is the free $S$-semimodule $S^{(G)}$ of all functions $f \colon G \to S$ with finite support, where the multiplication $f * g \colon G \to S$ is defined, for $f, g \in S^{(G)}$ and $\gamma \in G$, by the convolution
\[ (f * g)(\gamma) \defeq \sum_{\alpha \beta = \gamma} f(\alpha) g(\beta) \,. \]

Next, we recollect the notion of an ample topological groupoid.
A \emph{groupoid}~$\mcG$ is the set of morphisms in a small category in which every morphism is invertible.  We identify the objects $\mcGo \subseteq \mcG$ with their identity morphisms, called \emph{units}, and denote the source and the range maps by $s, r \colon \mcG \to \mcGo$.
Moreover, let $\mcG^{(2)} \defeq \{ (\alpha, \beta) \in \mcG^2 \mid r(\beta) = s(\alpha) \}$ be the set of composable pairs.

A \emph{topological groupoid} is a groupoid~$\mcG$ equipped with a topology such that the composition map $\mcG^{(2)} \to \mcG$, $(\alpha, \beta) \mapsto \alpha \beta$ and the inverse map $\mcG \to \mcG$, $\gamma \mapsto \gamma^{-1}$ are continuous.
The topological groupoid~$\mcG$ is called \emph{étale} if the unit space $\mcGo$ is locally compact Hausdorff and the source and range maps $s, r \colon \mcG \to \mcGo$ are local homeomorphisms.

An étale groupoid~$\mcG$ is said to be \emph{ample} in case its unit space is totally disconnected.  An ample groupoid has a topological base of compact open bisections~\cite[Prop.~3.6]{s:agatdisa}, where a \emph{bisection} is a subset $U \subseteq \mcG$ such that $s|_U, r|_U$ are homeomorphisms onto the image.
We remark that a compact open bisection is not necessarily closed and thus its characteristic function may be discontinuous.
Also, the intersection of two compact open bisections is not necessarily compact; in fact, this holds if and only if the groupoid is Hausdorff, cf.~\cite[Prop.~3.7]{s:agatdisa}.

Now we are ready to define the Steinberg algebra of an ample groupoid~$\mcG$ over a commutative semiring~$S$.  Denote by $S^{\mcG}$ the $S$-module of all functions $f \colon \mcG \to S$ and by $\mcG^a$ the set of all compact open bisections in~$\mcG$.

\begin{defn}\label{def:steinberg}
Let~$\mcG$ be an ample groupoid and~$S$ a commutative semiring.  The \emph{Steinberg algebra} $A_S(\mcG)$ is defined as the $S$-submodule of~$S^{\mcG}$ generated by the characteristic functions~$1_U$ of compact open bisections~$U$, \textit{i.e.}, \[ A_S(\mcG) \defeq \big\{ \sum_{U \in F} s_U 1_U \mid s_U \in S ,\, F \subseteq \mcG^a \text{ finite} \big\} \,, \]
with multiplication $f * g$, for $f, g \in A_S(\mcG)$ and $\gamma \in \mcG$, given by the convolution \[ (f * g)(\gamma) \defeq \sum_{\substack{r(\beta) = s(\alpha)\\ \alpha \beta = \gamma}} f(\alpha) g(\beta) \,. \]
\end{defn}

Regarding the well-definedness of this convolution and the algebraic properties of the Steinberg algebra, we have the following result.

\begin{prop}[{cf.~\cite[Prop.~4.5, 4.6]{s:agatdisa}}]\label{convoprod}
Let~$\mcG$ be an ample groupoid and~$S$ a commutative semiring.  Then there holds:
\begin{enumerate}[\quad \upshape (1)]
\item $f \ast g \in A_S(\mcG)$ for all $f, g \in A_S(\mcG)$;
\item $1_U \ast 1_V = 1_{U V}$ for any $U, V \in \mcG^a$, thus
  $1_U \ast 1_V = 1_{U \cap V}$ if also $U, V \subseteq \mcGo$;
\item  $1_{U^{-1}}(\gamma) = 1_U(\gamma^{-1})$ for any $U \in \mcG^a$ and $\gamma \in \mcG$;
\item $A_S(\mcG)$, equipped with the convolution, is an $S$-algebra.
\end{enumerate}
\end{prop}

\begin{proof}
Items (1) to (3) are proved similarly as in the proof of \cite[Prop.~4.5]{s:agatdisa}.  For item (4), it is sufficient to show the associativity of convolution.  However, this is a straightforward by using item (2) (or the proof of \cite[Prop.~2.4]{r:tgatlpa}).
\end{proof}

Note that the Steinberg algebra $A_S(\mcG)$ in general does not have an identity element.  In fact, it is unital if and only if the unit space $\mcGo$ is compact, in which case the identity element is given by $1_{\mcGo}$, cf.~\cite[Prop.~4.11]{s:agatdisa}.

In the following we study Steinberg algebras over additively idempotent commutative semirings and over the Boolean semifield in more detail.

\begin{lem}\label{expresslem}
Let~$\mcG$ be an ample groupoid with a base~$\mathcal B$ of compact open bisections, and $S$ an additively idempotent commutative semiring.  Then $A_S(\mcG) = \on{Span}_S\{ 1_U \mid U \text{ is compact open in } \mcG\} = \on{Span}_S \{ 1_B \mid B \in \mathcal B \}$.  If moreover the semiring~$S$ is the Boolean semifield~$\B$, then
\[ A_{\B}(\mcG) = \{ 1_U \mid U \text{ is compact open in } \mcG \} \,. \]
\end{lem}

\begin{proof} Obviously, there holds
\[ \on{Span}_S \{ 1_B \mid B \in \mathcal B \} \subseteq A_S(\mcG) \subseteq\on{Span}_S \{ 1_U \mid U \text{ is compact open in } \mcG \} . \]
Let~$U$ be a compact open subset of~$\mcG$.  Since~$\mathcal B$ is a base for the topology on~$\mcG$, there exist $B_1, \dots, B_k \in \mathcal B$ such that $U = B_1 \cup \ldots \cup B_k$.  Because the semiring~$S$ is additively idempotent, we readily obtain that
\[ 1_U = 1_{B_1 \cup \ldots \cup B_k} = 1_{B_1} + \ldots + 1_{B_k} \in \on{Span}_S \{ 1_B \mid B \in \mathcal B \} . \]
This shows the first statement of the lemma.  Now if the semiring~$S$ is the Boolean semifield~$\B$, it is clear that $\on{Span}_{\B} \{ 1_U \mid U \text{ is compact open in } \mcG \}$ equals the set $\{ 1_U \mid U \text{ is compact open in } \mcG \}$, thus concluding the proof.
\end{proof}

\begin{rem}\label{SAg-Bool-rem}
Let~$\mcG$ be an ample groupoid and let~$S$ be an additively idempotent commutative semiring.
  
(1) For all compact open subsets~$U$, $V$ of~$\mcG$ there holds \[ 1_U * 1_V = 1_{U V} \,, \] extending Proposition~\ref{convoprod} (2).
Indeed, by additive idempotency, for $\gamma \in \mcG$ the value $(1_U * 1_V)(\gamma) = \sum_{\alpha \beta = \gamma} 1_U(\alpha) 1_V(\beta)$ equals~$1$ iff there exists $\alpha \in U$ and $\beta \in V$ with $r(\beta) = s(\alpha)$ and $\alpha \beta = \gamma$, \textit{i.e.}, iff $\gamma \in U V$.

(2) If moreover the semiring~$S$ is the Boolean semifield~$\B$, there is a natural bijective correspondence between the Steinberg algebra $A_{\B}(\mcG)$ and the collection of all compact open subsets of~$\mcG$, given by $1_U \mapsto U$ where $U \subseteq \mcG$ is a compact open subset, cf.\ Lemma~\ref{expresslem}.
Under this bijection the Steinberg algebra operations correspond to the set operations \[ U + V \defeq U \cup V \,, \qquad U * V \defeq U V \] for any compact open subsets~$U$, $V$ of~$\mcG$.
\end{rem}

The following result provides us with a universal property of Steinberg algebras over the Boolean semifield~$\B$.

\begin{cor}\label{univproperty}
Let~$\mcG$ be an ample groupoid, and let~$B$ be a $\B$-algebra containing a family of elements $\{t_U \mid U \text{ is a compact open subset of } \mcG \}$ satisfying:
\begin{enumerate}[\quad \upshape (1)]
\item $t_{\varnothing} = 0$;
\item $t_U + t_V = t_{U \cup V}$ and $t_U \ast t_V = t_{U V}$ for all compact open subsets~$U$ and~$V$.
\end{enumerate}
Then there is a unique $\B$-algebra homomorphism $\pi \colon A_{\B}(\mcG) \longrightarrow B$ such that $\pi(1_U) = t_U$ for all compact open subsets~$U$.
\end{cor}

\begin{proof} By Lemma~\ref{expresslem}, we have \[ A_{\B}(\mcG) = \{ 1_U \mid U \text{ is compact open in } \mcG \} . \]
Using Remark~\ref{SAg-Bool-rem} (2), it is straightforward to show that the map $\pi \colon A_{\B}(\mcG) \longrightarrow B$, $1_U \longmapsto t_U$ for any compact open subset $U \subseteq \mcG$, is a $\B$-algebra homomorphism, thus finishing the proof. \end{proof}

We conclude this section by computing Steinberg algebras of ample groupoids with finite unit space, extending~\cite[Prop.~3.1]{s:ccoegawatlpaaisa} to commutative semirings.
Before doing so, we need to fix some useful notations for groupoids.  Let~$\mcG$ be a groupoid and let $D$, $E$ be subsets of $\mcGo$.  Define
\[ \mcG_D \defeq \{ \gamma \in \mcG \mid s(\gamma) \in D \} \,,\ \mcG^E \defeq \{\gamma \in \mcG \mid r(\gamma) \in E \} ~\text{ and }~ \mcG_D^E \defeq \mcG_D \cap \mcG^E \,. \]
For $u, v \in \mcGo$ we also briefly denote $\mcG_u \defeq \mcG_{\{ u \}}$, $\mcG^v \defeq \mcG^{\{ v \}}$ and $\mcG_u^v \defeq \mcG_u \cap \mcG^v$.
For a unit $u \in \mcGo$ the group $\mcG_u^u = \{ \gamma \in \mcG \mid s(\gamma) = u = r(\gamma) \}$ is called its \emph{isotropy group}.  The \emph{isotropy subgroupoid} of~$\mcG$ is $\on{Iso}(\mcG) \defeq \bigcup_{u \in \mcGo} \mcG^u_u$.

Consider the equivalence relation on~$\mcGo$, defined by $u \sim v$ if there exists some $g \in \mcG$ such that $s(g) = u$ and $r(g) = v$, whose equivalence classes are called \emph{orbits}.
It is well-known that, up to conjugation in~$\mcG$ (and hence isomorphism), the isotropy group of~$u$ depends only on the orbit of~$u$. Thus we may speak unambiguously of the isotropy group of the orbit.

\begin{prop}
Let~$S$ be a commutative semiring and~$\mcG$ an ample groupoid with~$\mcGo$ finite.  Let $\mcO_1, \dots, \mcO_k$ be the orbits of~$\mcGo$, with~$\mcO_i$ of size $n_i \defeq \vert \mcO_i \vert$ and having isotropy group~$G_i$, for $i = 1, \dots, k$.  Then there holds \[ A_S(\mcG) \,\cong\, \prod^k_{i=1} M_{n_i}(S G_i) \,, \]
where $M_{n_i}(S G_i)$ denotes the $n_i \!\times\! n_i$-matrices over the group algebra $S G_i$.
\end{prop}

\begin{proof}
The proof is essentially based on the one of \cite[Prop.~3.1]{s:ccoegawatlpaaisa}.  Since the unit space~$\mcGo$ is both Hausdorff and finite, it is discrete.
Then, since~$s$ is a local homeomorphism, also $\mcG$~is discrete.  By \cite[Ex.~2.8 (3)]{nz:osaohagocs}, $A_S(\mcG)$ is exactly the free $S$-algebra having basis~$\mcG$ and whose product extends that of~$\mcG$, where we interpret undefined products as~$0$.
For every~$i$, let $\mcG_i \defeq s^{-1}(\mcO_i) = r^{-1}(\mcO_i)$.  We then have that~$\mcG_i$ is an ample subgroupoid of~$\mcG$ and $\mcG = \mcG_1 \sqcup \dots \sqcup \mcG_k$, hence $A_S(\mcG) \cong \prod^k_{i=1} A_S(\mcG_i)$.

Now it suffices to show that $A_S(\mcG_i) \cong M_{n_i}(S G_i)$.  To this end, we fix an element $u_i \in \mcO_i$, and choose, for each $v \in \mcO_i$, an element $g_v\in \mcG_i$ such that $s(g_v) = u_i$ and $r(g_v) = v$.
Notice that $G_i \cong \mcG^{u_i}_{u_i}$ as groups.  We define a map $\phi \colon A_S(\mcG_i) \longrightarrow M_{n_i}(SG_i)$ by $\phi(g) \defeq g^{-1}_{r(g)}g g_{s(g)} E_{r(g) s(g)}$ for all $g \in \mcG_i$, where we index
the rows and columns of matrices by $\mcO_i$, and  $E_{r(g) s(g)}$ is the matrix unit in $M_{n_i}(SG_i)$.  It is not hard to see that~$\phi$ is an isomorphism, as desired.
\end{proof}

\section{Congruence-simpleness of Steinberg algebras}\label{sec:simpleness}

The main goal of this section is to present a description of the congruence-simple Steinberg algebras $A_S(\mcG)$ of second-countable ample groupoids $\mcG$ over a semifield~$S$, which extends the well-known description when
the ground semifield~$S$ is a field (see \cite[Th.~3.14]{cepss:soaatnhg} and \cite[Th.~4.16]{ss:soisaega}).

We begin by recalling some important notations for a groupoid~$\mcG$.
A subset~$D$ of the unit space~$\mcGo$ is called \emph{invariant} if $s(\gamma) \in D$ implies $r(\gamma) \in D$ for all $\gamma \in \mcG$; equivalently, $D = \{ r(\gamma) \mid s(\gamma)\in D\} = \{s(\gamma) \mid r(\gamma) \in D\}$.  Also, $D$ is invariant if and only if its complement is invariant.  Moreover, let $\mcT$ denote the set of units with trivial isotropy group; this set $\mcT$ is invariant.  The groupoid~$\mcG$ is said to be \emph{principal} if $\mcT = \mcGo$, \textit{i.e.}, if $\on{Iso}(\mcG) = \mcGo$.

\begin{defn}[{\cite[Def.~2.1]{bcfs:soaateg}}]
Let~$\mcG$ be an ample groupoid.  We say that~$\mcG$ is \emph{minimal} if~$\mcGo$ has no nontrivial open invariant subsets.  We say that~$\mcG$ is \emph{topologically principal} if the set $\mcT$ is dense in $\mcGo$, and we call~$\mcG$ \emph{effective} if the interior of $\on{Iso}(\mcG)$ is equal to $\mcGo$.
\end{defn}

Notice that effective groupoids are related to topologically principal groupoids.  Any effective second-countable ample groupoid is topologically principal (see \cite[Prop.~3.6]{r:csica}).  Conversely, any Hausdorff ample groupoid being topologically principal is effective (see~\cite[Lem.~3.1]{bcfs:soaateg}).  However, this fact is not generally true in the non-Hausdorff case (see \cite[Lem.~5.1]{cepss:soaatnhg}).

We now describe necessary conditions for Steinberg algebras of ample groupoids over semifields to be congruence-simple, the argument following the proof of \cite[Prop.~3.2]{nz:osaohagocs}.

\begin{prop}\label{Neccondprop}
If the Steinberg algebra $A_S(\mcG)$ of an ample groupoid~$\mcG$ over a semifield~$S$ is congruence-simple, then there holds:
\begin{enumerate}[\quad \upshape (1)]
\item $S$~is either a field or the Boolean semifield $\B$;
\item $\mcG$~is both minimal and effective.
\end{enumerate}
\end{prop}

\begin{proof}
First, we note that since~$S$ is a semifield, by \cite[Prop.~4.34]{g:sata}, $S$~is either a field or zerosumfree.
If~$S$ is a field, then the proposition follows from \cite[Th.~3.5]{s:spasoegawatisa}.  Henceforth we consider the case when~$S$ is a zerosumfree semifield. \medskip

(1) Since~$S$ is zerosumfree, there exists a surjective semiring homomorphism $\pi \colon S \longrightarrow \B$ such that $\pi(0) =0$ and $\pi(s) = 1$ for all $0 \ne s \in S$.
For any compact open subset~$U$ of~$\mcG$, we denote by~$t_U$ the characteristic function $\mcG \longrightarrow \B$ of~$U$.  We claim that the formula \[ \phi \colon A_S(\mcG) \longrightarrow A_{\B}(\mcG) \,, \quad \sum_{U \in F} a_U 1_U \longmapsto \sum_{U \in F} \pi(a_U) t_U \,, \]
is well-defined on linear combinations of indicator functions, where~$F$ is a finite set of compact open bisections.  Indeed, assume that \[ f \defeq \sum_{U \in F} a_U 1_U = \sum_{V \in H} b_V 1_V \eqdef g , \]
where each~$a_U$ and~$b_V$ is nonzero in~$S$, and each of~$F$ and~$H$ is a finite set of compact open bisections.  Then, since~$S$ is zerosumfree, and $f = g$, we have $\bigcup_{U \in F} U = \bigcup_{V \in H} V$,
so $\phi(f) = \sum_{U \in F} t_U = t_{\cup_{U \in F}} = t_{\cup_{V \in H}} = \sum_{V \in H} t_V = \phi(g)$, as desired.  By using Proposition~\ref{convoprod} (2), it is straightforward to show that~$\phi$ is a semiring homomorphism.

Assume that there exists a nonzero element $s \in S$ such that $s \ne 1$.  Fix a nonempty compact open bisection~$U$ of $\mcG$.  We then have $s 1_U \ne 1_U$ and $\phi(s 1_U) = \pi(s) t_U = t_U = \phi(1_U)$.  Therefore, the map~$\phi$ defines a nontrivial congruence and $A_S(\mcG)$ is not congruence-simple, a contradiction.  Thus we obtain that $s = 1$ for all $0 \ne s\in S$, and~$S$ is exactly the Boolean semifield~$\B$, as desired. \medskip

(2) Assume that~$\mcGo$ contains a nontrivial open invariant subset~$V$.  Let $D \defeq \mcGo \setminus V$.  We then have that~$\mcG_D$ coincides with the \emph{restriction} \[ \mcG|_{D} \defeq \{ \gamma \in \mcG \mid s(\gamma), \, r(\gamma) \in D \} \] of~$\mcG$ to~$D$.
Thus~$\mcG_D$ is a topological subgroupoid of~$\mcG$ with the relative topology, and its unit space is~$D$.  Since~$D$ is closed in~$\mcGo$, and the map $s \colon \mcG \longrightarrow \mcGo$ is continuous,
$\mcG_D = s^{-1}(D)$ is closed in $\mcG$, and so $U \cap \mcG_D$ is a compact open bisection of~$\mcG_D$ for any compact open bisection~$U$ of~$\mcG$.  This implies that~$\mcG_D$ is an ample groupoid.

For any compact open subset~$U$ of~$\mcG$, we denote by~$t_U$ the characteristic function $\mcG_D \longrightarrow \B$ of $U \cap \mcG_D$.  It is clear that the collection $\{ t_U \mid U$~is a compact open in~$\mcG \}$ of elements of $A_{\B}(\mcG_D)$ satisfies the conditions (1) and (2) of Corollary~\ref{univproperty}, by which there exists a $\B$-algebra homomorphism $\phi \colon A_{\B}(\mcG) \longrightarrow A_{\B}(\mcG_D)$ such that $\phi(1_U) = t_U$ for all compact open subsets~$U$ of~$\mcG$.
Since~$V$ is a nontrivial open subset of~$\mcGo$, there exist nonempty compact open subsets~$U_1$ and~$U_2$ of~$\mcGo$ such that $U_1 \subseteq V$ and $U_2 \cap D \ne \varnothing$.  This implies that $\phi(1_{U_1}) =t_{U_1 \cap \mcG_D} = 0$ (since $U_1 \cap \mcG_D = \varnothing$) and $\phi(1_{U_2}) = t_{U_2 \cap \mcG_D} \ne 0$, so $\phi$ is a nonzero homomorphism, but not injective.  Therefore, $A_{\B}(\mcG)$ is not congruence-simple, a contradiction, whence $\mcG$ is minimal. \medskip

We next show that~$\mcG$ is effective.  Denote by $F(\mcGo)$ the free $\B$-semimodule with basis~$\mcGo$.  Let~$U$ be a compact open bisection of~$\mcG$, and observe that $s(\alpha) \longmapsto r(\alpha)$ determines a homeomorphism from $s(U)$ to $r(U)$.  We define a map $f_U \colon \mcGo \longrightarrow F(\mcGo)$ by
\[ f_U(x) = \begin{cases}
r(\alpha) & \text{if } x = s(\alpha) \text{ and } \alpha \in U , \\
0 & \text{otherwise}, \end{cases} \]
for all $x \in \mcGo$.  By the universal property
of the free $\B$-semimodule $F(\mcGo)$, there exists an element $t_U \in \on{End}_{\B}(F(\mcGo))$ extending~$f_U$.

For any compact open subset~$B$ of~$\mcG$, define $t_B \defeq \sum_{U \in F} t_U \in \on{End}_{\B}(F(\mcGo))$, where $B = \bigcup_{U \in F} U$, with~$F$ a finite set of compact open bisections of~$\mcG$.  Indeed, for any $x \in \mcGo$, we have
\[ \sum_{U \in F} t_U(x) = \sum_{U \in F} \sum_{\alpha \in U} \{ r(\alpha) \mid x = s(\alpha) \} = \sum_{\alpha \in \bigcup_{U \in F} U} \{ r(\alpha) \mid x = s(\alpha) \} \,. \]  In particular, $t_B$ only depends on $B = \bigcup_{U \in F} U$ and is well-defined.

It is easy to see that the collection $\{t_B \mid B \text{ is compact open in } \mcG \}$ of elements of $\on{End}_{\B}(F(\mcGo))$ satisfies the conditions (1) and (2) of Corollary~\ref{univproperty}, by which there exists a $\B$-algebra homomorphism $\psi \colon A_{\B}(\mcG) \longrightarrow \on{End}_{\B} (F(\mcGo))$ such that $\psi(1_B) = t_B$ for all compact open subsets~$B$ of~$\mcG$.  The homomorphism~$\psi$ is nonzero because $t_U \ne 0$ for any nonempty compact open bisection~$U$ of~$\mcG$.

Now suppose that~$\mcG$ is not effective.  Then there is a nonempty compact open bisection $U \subseteq \mcG$ such that $U \nsubseteq \mcGo$ and $s(\alpha) = r(\alpha)$ for all $\alpha \in U$.  It implies that $U \ne s(U)$ and $t_U = t_{s(U)}$, hence $1_U \ne 1_{s(U)}$ whereas \[ \psi(1_U) = t_U = t_{s(U)} =\psi(1_{s(U)}) , \]
showing that~$\psi$ is not injective.  Therefore, $A_{\B}(\mcG)$ is not congruence-simple, a contradiction, so that $\mcG$~is effective, concluding the proof.
\end{proof}

The next result, being a $\B$-algebra analog of \cite[Prop.~4.5]{bcfs:soaateg} and \cite[Prop.~3.4]{s:spasoegawatisa}, provides a criterion for minimal ample groupoids.  It plays an important role in the proof of the subsequent main result (Theorem~\ref{thm:simpleness}).

\begin{lem}\label{minimalcriterion}
An ample groupoid~$\mcG$ is minimal if and only if~$1_V$ generates $A_{\B}(\mcG)$ as an ideal for all nonempty compact open subsets~$V$ of~$\mcGo$.
\end{lem}

\begin{proof}
($\Longrightarrow$). Assume that~$\mcG$ is a minimal ample groupoid.  Let~$V$ be a nonempty compact open subset of~$\mcGo$, and~$I$ an ideal of $A_{\B}(\mcG)$ generated by~$1_V$.  If~$U$ is any compact open bisection of~$\mcG$, we must prove that $1_U \in I$.  Let $K \defeq r(U) \subseteq \mcGo$.  We then have that~$K$ is a compact open subset of~$\mcGo$.  Since $s(\mcG^V)$ is a nonempty open invariant set, and~$\mcG$ is minimal, $s(\mcG^V) = \mcGo$, and hence $K \subseteq s(\mcG^V)$.  Thus, for any $u \in K$, there exists an element $\alpha_u \in \mcG$ such that $s(\alpha_u) = u$ and $r(\alpha_u) \in V$.  For each $u \in K$, let $B_u$ be a compact open bisection of~$\mcG$ containing~$\alpha_u$ such that $r(B_u) \subseteq V$ and $s(B_u) \subseteq K$.  We then have $1_{s(B_u)} = 1_{B^{-1}_u} \ast 1_V \ast 1_{B_u} \in I$.  Since~$K$ is compact, there exists a finite subset $\{ u_1, \dots, u_n \}$ of~$K$ such that $\{ s(B_{u_i}) \mid i = 1, \dots, n \}$ covers~$K$.  Therefore, $1_K = \sum_{i=1}^n 1_{s(B_{u_i})} \in I$, and so $1_U = 1_K \ast 1_U \in I$.  It follows that $I = A_{\B}(\mcG)$. \medskip
	
($\Longleftarrow$). Suppose that~$\mcG$ is not minimal.  Let~$U$ be a nontrivial open invariant subset of~$\mcGo$ and let $D \defeq \mcGo \setminus U$.  Then~$\mcG_D$ coincides with the restriction $\mcG|_D \defeq \{ \gamma \in \mcG \mid s(\gamma), \, r(\gamma) \in D\}$ of~$\mcG$ to~$D$.
As shown in the proof of Proposition~\ref{Neccondprop}, $\mcG_D$ is an ample groupoid with unit space~$D$, and there is a nonzero $\B$-algebra homomorphism $\phi \colon A_{\B}(\mcG) \longrightarrow A_{\B}(\mcG_D)$ such that $\phi(1_B) = 1_{B \cap \mcG_D}$ for all compact open subsets~$B$ of $\mcG$.
This implies that $\on{Ker}(\phi) \defeq \phi^{-1}(0)$ is a proper ideal of $A_{\B}(\mcG)$.  Since~$U$ is a nontrivial open subset of~$\mcGo$, there exists a nonempty compact open subset~$V$ of~$\mcGo$ such that $V \subseteq U$.  We then have $\phi(1_V) = t_{V \cap \mcG_D} = 0$ (since $V \cap \mcG_D = \varnothing$), and so $1_V \in \on{Ker}(\phi) \setminus \{ 0 \}$, whence by our hypothesis $\on{Ker}(\phi) = A_{\B}(\mcG)$, a contradiction, thus finishing the proof.
\end{proof}

Recall that $\mcT = \{ u \in \mcGo \mid \mcG_u^u = \{ u \} \}$ is the (invariant) set of units in a groupoid~$\mcG$ that have trivial isotropy group.

\begin{defn}
Let~$S$ be a semiring and~$\mcG$ an ample groupoid.  We denote by $\equiv_{A_S(\mcG)}$ the binary relation on $A_S(\mcG)$ defined by: for all $f$ and $g \in A_S(\mcG)$, 
\[ f \equiv_{A_S(\mcG)} g ~\text{ iff }~ f(\alpha) = g(\alpha) \text{ for all } \alpha \in \mcG \text{ such that } s(\alpha) , r(\alpha) \in \mcT . \]
\end{defn}

We have the following useful fact.

\begin{lem}\label{lem:cong1}
For a commutative semiring~$S$ and an ample groupoid $\mcG$, the relation $\equiv_{A_S(\mcG)}$ is a congruence on $A_S(\mcG)$.
\end{lem}

\begin{proof}
It is obvious that $\equiv_{A_S(\mcG)}$ is an equivalence relation on $A_S(\mcG)$ satisfying $s f \equiv_{A_S(\mcG)} s g$ and $f + h \equiv_{A_S(\mcG)} g + h$ for all $s \in S$ and $f, g, h \in A_S(\mcG)$ with $f \equiv_{A_S(\mcG)} g$.

Let $f, g, h \in A_S(\mcG)$ with $f \equiv_{A_S(\mcG)} g$ and $\gamma \in \mcG$ with $s(\gamma), r(\gamma) \in \mcT$.  Notice that whenever $\gamma = \alpha \beta$ for some $\alpha, \beta \in \mcG$, we have $s(\beta) = s(\gamma) \in \mcT$ and $r(\alpha) = r(\gamma) \in \mcT$, as well as $r(\beta) = s(\alpha) \in \mcT$, since $\mcT$ is an invariant subset of $\mcGo$.  From $f \equiv_{A_S(\mcG)} g$ and these remarks, we obtain that
\begin{gather*}
  (f \ast h) (\gamma) = \sum_{\substack{r(\beta) = s(\alpha)\\ \alpha \beta = \gamma}}\! f(\alpha) h(\beta) = \sum_{\substack{r(\beta) = s(\alpha)\\ \alpha \beta = \gamma}}\! g(\alpha) h(\beta) = (g \ast h) (\gamma) , \\
  (h \ast f) (\gamma) = \sum_{\substack{r(\beta) = s(\alpha)\\ \alpha \beta = \gamma}}\! h(\alpha) f(\beta) = \sum_{\substack{r(\beta) = s(\alpha)\\ \alpha \beta = \gamma}}\! h(\alpha) g(\beta) = (h \ast g) (\gamma) ,
\end{gather*}
so that $f \ast h \equiv_{A_S(\mcG)} g \ast h$ and $h \ast f \equiv_{A_S(\mcG)} h \ast g$, as desired.
\end{proof}

The following result, being a “congruence” analog of \cite[Prop.~4.1]{n:goegasa}, plays an important role in the proof of our main result below.

\begin{prop}\label{cong-simp}
For a topologically principal and minimal ample groupoid~$\mcG$ over a semifield~$S$, the quotient algebra $A_S(\mcG) /_{\equiv_{A_S(\mcG)}}$ is congruence-simple if and only if~$S$ is either a field or the Boolean semifield~$\B$.
\end{prop}

\begin{proof}
($\Longrightarrow$). Assume that $A_S(\mcG) /_{\equiv_{A_S(\mcG)}}$ is congruence-simple.  Since~$S$ is a semifield, $S$~is either a field or zerosumfree.  Consider the case when~$S$ is a zerosumfree semifield.
Then, there exists a semiring homomorphism $\pi \colon S \longrightarrow \B$ given by $\pi(0) = 0$ and $\pi(s) = 1$ for all $0 \ne s \in S$.
We define a binary relation~$\rho$ on $A_S(\mcG)$ as follows: for all $f = \sum_{B \in F} a_B 1_B$ and $g = \sum_{C \in H} b_C 1_C$, where each of~$F$ and~$H$ is a finite set of compact open bisections of $\mcG$,
$f \,\rho\, g$ if and only if $\sum_{B \in F} \pi(a_B) 1_B(\alpha) = \sum_{C \in H} \pi(b_C) 1_C(\alpha)$ in $A_{\B}(\mcG)$ for all $\alpha \in \mcG$ such that $r(\alpha), s(\alpha) \in \mcT$.
Following the same argument as in the proof of Proposition~\ref{Neccondprop}\,(1), the relation~$\rho$ is well-defined.  Moreover, similar to the proof of Lemma~\ref{lem:cong1}, we have that~$\rho$ is a congruence on $A_S(\mcG)$ containing $\equiv_{A_S(\mcG)}$.

Since $\mcG$ is topologically principal, $(1_U, 0) \notin \rho$ for any nonempty compact open subset~$U$ of $\mcGo$, and so $\rho \ne A_S(\mcG)^2$.
Because $A_S(\mcG) /_{\equiv_{A_S(\mcG)}}$ is congruence-simple, $\equiv_{A_S(\mcG)}$ is a maximal congruence on $A_S(\mcG)$, and so $\rho$ equals $\equiv_{A_S(\mcG)}$.  Let~$s$ be a nonzero element in~$S$ and~$U$ a nonempty compact open subset of $\mcGo$.
It is obvious that $(s 1_U, 1_U) \in \rho$, so that $s 1_U \equiv_{A_S(\mcG)} 1_U$, which means that $s 1_U(\alpha) = 1_U(\alpha)$ in $A_S(\mcG)$ for all $\alpha \in \mcG$ such that $r(\alpha), s(\alpha) \in \mcT$.  Since~$\mcG$ is topologically principal, there exists an element $u \in U \cap \mcT$.  We then have $s = s 1_U(u) = 1_U(u) = 1$, and thus $S = \B$. \medskip

($\Longleftarrow$). If~$S$ is a field, then $A_S(\mcG) /_{\equiv_{A_S(\mcG)}} \cong A_S(\mcG)/I$, where $I = \{ f \in A_S(\mcG) \mid f \equiv_{A_S(\mcG)} 0 \}$.  Therefore, the statement follows from \cite[Prop.~4.1]{n:goegasa}.

Consider the case $S = \B$.  It is sufficient to prove that $\equiv_{A_{\B}(\mcG)}$ is a maximal congruence on $A_{\B}(\mcG)$.  We first note that $1_U \not\equiv_{A_{\B}(\mcG)} 0$ for all nonempty compact open subsets~$U$ of $\mcGo$, and so $\equiv_{A_{\B}(\mcG)}$ is not $A_{\B}(\mcG)^2$.

Assume that~$\rho$ is a congruence on $A_{\B}(\mcG)$ such that ${\equiv_{A_{\B}(\mcG)}} \subset \rho$.  Then, there exist two elements~$f$ and~$g$ in $A_{\B}(\mcG)$ such that $(f, g) \in \rho$ and $f \not\equiv_{A_{\B}(\mcG)} g$, \textit{i.e.}, $f(\alpha) \ne g(\alpha)$ for some $\alpha \in \mcG$ with $r(\alpha), s(\alpha)\in \mcT$.
Let~$U$ be a compact open bisection of~$\mcG$ containing $\alpha^{-1}$, and let $u \defeq r(\alpha) = \alpha \alpha^{-1} \in \mcT$.  We have
\[ f \ast 1_U(u) = \sum_{\beta \gamma = u} f(\beta) 1_U(\gamma) = f(\alpha) 1_U(\alpha^{-1}) = f(\alpha) , \]
since the unique $\gamma \in U$ such that $s(\gamma) = s(u) = u$ is given by $\gamma = \alpha^{-1}$, and then $\beta = \alpha$.  Similarly, $g \ast 1_U(u) = g(\alpha)$, and so $f \ast 1_U(u) \neq g \ast 1_U(u)$.

Thus, there exist two elements $\theta, \psi \in A_{\B}(\mcG)$ such that $(\theta, \psi) \in \rho$ with $\theta(u)\neq \psi(u)$ for some $u \in \mcT$ (we may take $\theta = f \ast 1_U$ and $\psi = g\ast 1_U$ as above).  Without loss of generality, one may assume that $\theta(u) = 0$ and $\psi(u) =1$.
Write $\theta = \sum_{B \in F} 1_B$ and $\psi = \sum_{C \in H} 1_C$, where each of~$F$ and~$H$ is a finite set of compact open bisections of $\mcG$.  For each $D \in F \cup H$, choose a compact open neighbourhood $V_D \subseteq \mcGo$ of $u$ as follows:

\textit{Case 1:} $u\in D$.  Let $V_D$ be a compact open subset in $\mcGo$ such that $u \in V_D$ and $V_D \subseteq D \cap \mcGo$.  Then $V_D D V_D \subseteq D \cap \mcGo$.

\textit{Case 2:} there exists $\gamma \in D$ such that $r(\gamma) = u$ and $s(\gamma) \ne u$.  Since $\mcGo$ is Hausdorff, there are two compact open subsets $U$ and $W$ of $\mcGo$ such that $u = r(\gamma) \in U$, $s(\gamma) \in W$ and $U\cap W = \varnothing$.
Then $U \cap r(D)$ and $W \cap s(D)$ are nonempty compact open subsets of $\mcGo$, and the set $D' \defeq (U \cap r(D)) D (W \cap s(D))$ is a compact open bisection of~$\mcG$.  More\-over, $\gamma \in D' \subseteq D$ and $r(D') \cap s(D') = \varnothing$, and we let $V_D \defeq r(D')$.  Then we have $u \in V_D$ and $V_D D V_D = \varnothing$.

\textit{Case 3:} there exists $\gamma \in D$ such that $s(\gamma) = u$ and $r(\gamma) \ne u$.  Since $\mcGo$ is Hausdorff, there are two compact open subsets $U$ and $W$ of $\mcGo$ such that $u = s(\gamma) \in U$, $r(\gamma) \in W$ and $U\cap W = \varnothing$.
Then $U \cap s(D)$ and $W \cap r(D)$ are nonempty compact open subsets of $\mcGo$, and the set $D' \defeq (W \cap r(D)) D (U \cap s(D))$ is a compact open bisection of~$\mcG$.  More\-over, $\gamma \in D' \subseteq D$ and $r(D') \cap s(D') = \varnothing$, and we let $V_D \defeq s(D')$.  Then we have $u \in V_D$ and $V_D D V_D = \varnothing$.

\textit{Case 4:} $u \notin r(D) \cup s(D)$.  Since $\mcGo$ is Hausdorff, and $r(D)$ and $s(D)$ are compact open subsets of $\mcGo$, they are closed in $\mcGo$, so $\mcGo \setminus (r(D) \cup s(D))$ is an open subset of $\mcGo$ containing $u$.
Let $V_D$ be a compact open subset of $\mcGo$ such that $u \in V_D \subseteq \mcGo \setminus (r(D) \cup s(D))$.  We then have $V_D D V_D = \varnothing$. \medskip

Let $V \defeq \bigcap_{D \in F \cup H} V_D$.  Then $V$ is a compact open subset of $\mcGo$ containing~$u$.  Furthermore, since $\theta(u) = 0$, \textit{i.e.}, $u \notin \bigcup_{B \in F} B$, and by construction of $V$, we have $1_V \ast \theta \ast 1_V = \sum_{B \in F} 1_{V B V} = 0$.
Since $\psi(u) =1$, we have $u \in D$ for some $D \in H$.  By construction of $V$, the set $U \defeq V D V$ is a compact open subset of $\mcGo$ containing $u$, and so $1_V \ast \psi \ast 1_V(u) = \sum_{C \in H} 1_{V C V}(u) = 1$.  Let $h \defeq 1_V \ast \psi \ast 1_V$.  We then note that $h \ne 0$ and $(h, 0) = (1_V \ast \psi \ast 1_V ,\, 1_V \ast \theta \ast 1_V) \in \rho$, and so $(1_U, 0) = (1_U \ast h ,\, 1_U \ast 0) \in \rho$.

Let us consider the ideal $I \defeq \{ f \in A_{\B}(\mcG) \mid (f, 0) \in \rho \}$ of $A_{\B}(\mcG)$.
From the observation above, $I$~contains a nonzero element~$1_U$ with $\on{supp}(1_U) = U \subseteq \mcGo$.  Since~$\mcG$ is minimal, we infer $I = A_{\B}(\mcG)$ from Lemma~\ref{minimalcriterion}.  It immediately follows that $\rho = A_{\B}(\mcG)^2$, whence $\equiv_{A_{\B}(\mcG)}$ is a maximal congruence on $A_{\B}(\mcG)$.  This concludes the proof.
\end{proof}

Let~$S$ be a commutative semiring, $\mcG$ an ample groupoid, and $f, g \in A_S(\mcG)$. Let \[ \on{Eq}(f, g) \defeq \{ \alpha\in \mcG \mid f(\alpha) = g(\alpha) \} . \]
We define a binary relation $\rho_{A_S(\mcG)}$ on $A_S(\mcG)$ as follows: $f \,\rho_{A_S(\mcG)}\, g$ if and only if $\on{Eq}(f, g)$ is dense in~$\mcG$.  We should mention that if~$S$ is a field, then $\rho_{A_S(\mcG)}$ is a congruence on $A_S(\mcG)$, by \cite[Prop.~3.7]{cepss:soaatnhg}.  The following result plays an important role in the proof of our main result below.

\begin{lem}\label{lem:cong2}
For an ample groupoid $\mcG$, the following statements hold:
\begin{enumerate}[\quad \upshape (1)]
\item $\rho_{A_{\B}(\mcG)}$ is a congruence on $A_{\B}(\mcG)$;
\item If, in addition, $\mcG$ is topologically principal, then $\rho_{A_{\B}(\mcG)}$ equals $\,\equiv_{A_{\B}(\mcG)}$.
\end{enumerate}
\end{lem}

\begin{proof} (1) From Remark~\ref{SAg-Bool-rem}\,(2), every element of $A_{\B}(\mcG)$ is of the form~$1_U$, where~$U$ is a compact open subset of~$\mcG$.  Then, if~$U$ and~$V$ are compact open subsets of~$\mcG$, notice that $(U \cup V) \setminus \overline{U \cap V} = \big( (U \cup V) \setminus (U \cap V) \big)^{\!\circ}$ and hence
\[ (1_U, 1_V) \in \rho_{A_{\B}(\mcG)} ~\Leftrightarrow~ \big( (U \cup V) \setminus (U \cap V) \big)^{\!\circ} = \varnothing ~\Leftrightarrow~ U \cup V \subseteq \overline{U \cap V} \,. \]

It is obvious that $\rho_{A_{\B}(\mcG)}$ is reflexive and symmetric.  Let $U$, $V$ and $W$ be compact open subsets of $\mcG$ such that $(1_U, 1_V) \in \rho_{A_{\B}(\mcG)}$ and $(1_V, 1_W) \in \rho_{A_{\B}(\mcG)}$, \textit{i.e.}, $U \cup V \subseteq \overline{U \cap V}$ and $V \cup W \subseteq \overline{V \cap W}$.
Suppose that $(1_U, 1_W) \notin \rho_{A_{\B}(\mcG)}$, then there exists a nonempty open set $A \subseteq (U \cup W) \setminus \overline{U \cap W}$.  In particular, $A \cap (U \cup W) \ne \varnothing$ and we may assume that $A \cap U \ne \varnothing$.
Since $A \cap U \subseteq U \cup V \subseteq \overline{U \cap V}$ we have a nonempty open set $B \defeq A \cap U \cap V \ne \varnothing$, and since $B \subseteq V \cup W \subseteq \overline{V \cap W}$ it follows that $A \cap U \cap V \cap W = B \cap V \cap W \ne \varnothing$.
In particular, $A \cap U \cap W \ne \varnothing$ and thus $A \cap \overline{U \cap W} \ne \varnothing$, which is a contradiction.
This shows that $\rho_{A_{\B}(\mcG)}$ is transitive, and therefore, $\rho_{A_{\B}(\mcG)}$ is an equivalence relation.

Next we claim that $\rho_{A_{\B}(\mcG)}$ is a congruence on $A_{\B}(\mcG)$.  It is, indeed, enough to show that the pairs $(1_U + 1_W ,\, 1_V + 1_W)$, $(1_U \ast 1_W ,\, 1_V \ast 1_W)$
and ($1_W \ast 1_U ,\, 1_W \ast 1_V)$ are in $\rho_{A_{\B}(\mcG)}$ for all compact open bisections~$W$ and compact open subsets~$U, V$ of $\mcG$ with $(1_U, 1_V) \in \rho_{A_{\B}(\mcG)}$.

We have $1_U + 1_W = 1_{U \cup W}$ and $1_V + 1_W = 1_{V \cup W}$, and since $U \cup V \subseteq \overline{U \cap V}$ we infer $U \cup V \cup W \subseteq (\overline{U \cap V}) \cup W \subseteq \overline{(U \cap V) \cup W} = \overline{(U \cup W) \cap (V \cup W)}$, which shows that $(1_U + 1_W ,\, 1_V + 1_W) \in \rho_{A_{\B}(\mcG)}$.

Now notice that $1_U \ast 1_W = 1_{U W}$ and $1_V \ast 1_W = 1_{V W}$, and suppose that $(1_{U W}, 1_{V W}) \notin \rho_{A_{\B}(\mcG)}$.  Then, since~$\mcG$ is ample, there exists a nonempty compact open bisection $A \subseteq (U W \cup V W) \setminus \overline{U W \cap V W}$ of~$\mcG$.  We then have $s(A) \subseteq s(W)$ and $A W^{-1} \subseteq (U W \cup V W) W^{-1} \subseteq U \cup V$.
The set $A W^{-1}$ is open and nonempty, since $A = A s(A) \subseteq A s(W) = A W^{-1} W$.  Hence, from $U \cup V \subseteq \overline{U \cap V}$ it follows that $A W^{-1} \cap U \cap V \ne \varnothing$, whence $A \cap (U W \cap V W) \ne \varnothing$, which is a contradiction.  This shows that $(1_U \ast 1_W ,\, 1_V \ast 1_W) \in \rho_{A_{\B}(\mcG)}$, and the proof for $(1_W \ast 1_U ,\, 1_W \ast 1_V) \in \rho_{A_{\B}(\mcG)}$ is similar. \medskip

(2) Since $\mcG$ is topologically principal, the algebra $A_S(\mcG)/_{\equiv_{A_S(\mcG)}}$ is congruence-simple by Proposition \ref{cong-simp} and thus the congruence $\equiv_{A_S(\mcG)}$ is maximal.  We claim that ${\equiv_{A_{\B}(\mcG)}} \subseteq \rho_{A_{\B}(\mcG)}$.
Indeed, let~$U$ and~$V$ be compact open subsets of~$\mcG$ such that $1_U \equiv_{A_{\B}(\mcG)} 1_V$, and let $A \subseteq \mcG$ be any nonempty open set.  Then $s(A)$ is a nonempty open set in~$\mcGo$, and hence $s(A) \cap \mcT \ne \varnothing$, since~$\mcG$ is topologically principal.
Since $1_U \equiv_{A_{\B}(\mcG)} 1_V$, it follows that $1_U(\alpha) = 1_V(\alpha)$ for some $\alpha \in A$ with $s(\alpha) \in \mcT$ (and thus $r(\alpha) \in \mcT$), and hence $\alpha \in \on{Eq}(1_U, 1_V) \cap A$.  This shows that $\on{Eq}(1_U, 1_V)$ is dense in~$\mcG$, and we have $(1_U, 1_V) \in \rho_{A_{\B}(\mcG)}$, as claimed.

Observe that $(1_U, 0) \notin \rho_{A_{\B}(\mcG)}$ for any nonempty compact open subset~$U$ of~$\mcGo$, whence $\rho_{A_{\B}(\mcG)} \ne A_{\B}(\mcG)^2$.  Since $\equiv_{A_S(\mcG)}$ is maximal and contained in~$\rho_{A_{\B}(\mcG)}$, we infer that the congruences are equal, as desired.
\end{proof}

Now we are able to present the main result of this section, which is a “semiring” analog of \cite[Th.~3.14]{cepss:soaatnhg} and \cite[Th.~4.16]{ss:soisaega},
characterizing the congruence-simple Steinberg algebras of second-countable ample groupoids over semifields.

\begin{thm}\label{thm:simpleness}
Let~$\mcG$ be a second-countable ample groupoid over a semifield~$S$.  Then the Steinberg algebra $A_S(\mcG)$ is congruence-simple if and only if the following conditions are satisfied:
\begin{enumerate}[\quad \upshape (1)]
\item $S$~is either a field or the Boolean semifield $\B$;
\item $\mcG$~is both minimal and effective;
\item $\on{Eq}(f, g)$ is not dense in~$\mcG$, for any two distinct elements~$f$ and~$g$ in $A_S(\mcG)$. 
\end{enumerate}
\end{thm}

\begin{proof}
($\Longrightarrow$). Items (1) and (2) follow from Proposition~\ref{Neccondprop}.  If~$S$ is a field, then item (3) follows from \cite[Th.~3.14]{cepss:soaatnhg}.  Consider the case when $S = \B$.
Since $\equiv_{A_{\B}(\mcG)}$ is a congruence by Lemma~\ref{lem:cong1}, and $A_S(\mcG)$ is congruence-simple, we infer that $\equiv_{A_{\B}(\mcG)}$ equals $\Delta_{A_{\B}(\mcG)}$.
Since~$\mcG$ is an effective second-countable ample groupoid, by \cite[Prop.~3.6]{r:csica} it is topologically principal.  Then, by Lemma~\ref{lem:cong2}, we readily obtain that $\rho_{A_{\B}(\mcG)} = \Delta_{A_{\B}(\mcG)}$, proving item (3).

($\Longleftarrow$). If~$S$ is a field, then the statement follows from \cite[Th.~3.14]{cepss:soaatnhg}.  Consider the case when $S = \B$.  Since~$\mcG$ is an effective second-countable ample groupoid, by \cite[Prop.~3.6]{r:csica} it is topologically principal.
From Lemma~\ref{lem:cong2}\,(2) we have that $\equiv_{A_{\B}(\mcG)}$ equals $\rho_{A_{\B}(\mcG)}$, while $\rho_{A_{\B}(\mcG)} = \Delta_{A_{\B}(\mcG)}$ because of item (3).  Then it follows from Proposition \ref{cong-simp} that the algebra $A_{\B}(\mcG) = A_{\B}(\mcG) /_{\equiv_{A_{\B}(\mcG)}}$ is congruence-simple, thus finishing the proof.
\end{proof}

In a topological space we say a subset~$B$ is a \emph{regular open}
set if~$B$ equals the interior of its closure.  The intersection of a
collection of regular open sets is again regular open but the same is
not true for unions; see, \textit{e.g.}, \cite[Chap.~10]{gh:itba} for
a detailed discussion of regular open sets.
By Theorem~\ref{thm:simpleness}, we obtain the following interesting
corollary, which provides sufficient conditions for Steinberg algebras
over the Boolean semifield~$\B$ to be congruence-simple, and being an
$\B$-algebra analog of \cite[Cor.~3.16]{cepss:soaatnhg}.

\begin{cor}\label{sim-cor}
Let $\mcG$ be a second-countable ample groupoid which satisfies the following two conditions:
\begin{enumerate}[\quad \upshape (1)]
\item $\mcG$ is minimal and effective;
\item every compact open subset of $\mcG$ is regular open.
\end{enumerate}
Then $A_{\B}(\mcG)$ is congruence-simple.
\end{cor}

\begin{proof}
Let~$f$ and~$g$ be two elements in $A_{\B}(\mcG)$.  By Lemma \ref{expresslem}, we have $f = 1_U$ and $g = 1_V$ for some compact open subsets~$U$ and~$V$ of~$\mcG$.  Suppose that $\on{Eq}(f, g)$ is dense in~$\mcG$, then we have $U \cup V \subseteq \overline{U \cap V}$.
However, the sets~$U$ and~$V$ are regular open by our hypothesis, and so $U \cap V$ is regular open.  It follows that $U \cup V \subseteq (\overline{U \cap V})^{\circ} = U \cap V$, so we infer that $U = V$ and therefore $f = g$.
Now applying Theorem~\ref{thm:simpleness}, we obtain that $A_{\B}(\mcG)$ is congruence-simple.
\end{proof}

\section{Examples based on self-similar graphs}\label{sec:examples}

In this section we establish semiring analogs of the main results introduced in \cite[Sec.~5]{cepss:soaatnhg}, which provide us with examples of congruence-simple Steinberg algebras of non-Hausdorff ample groupoids over the Boolean semifield~$\B$ (Theorem~\ref{sim-ssgalg}, Corollaries~\ref{sim-Katalg} and~\ref{sim-ssacalg}).
The examples are based on so-called self-similar graphs, which give rise to ample groupoids via the tight groupoid of an inverse semigroup associated to the self-similar graph.

We start by recalling the notion of a self-similar graph~\cite{ep:ssgautokanca}.  Consider a directed graph $E = (E^0, E^1, r, s)$ with vertex set~$E^0$, edge set~$E^1$ and source and range maps $s, r \colon E^1 \to E^0$.
A vertex $v \in E^0$ is called \emph{source} if $r^{-1}(v) = \varnothing$.  Moreover, denote by~$E^*$ the finite paths and by~$E^{\infty}$ the (one-sided) infinite paths in the graph~$E$.
We view paths from right to left, so that $\alpha \beta$ is defined if and only if $r(\beta) = s(\alpha)$, and $s(\alpha \beta) = s(\beta)$ and $r(\alpha \beta) = r(\alpha)$, for $\alpha, \beta \in E^*$.  Infinite paths $\xi \in E^{\infty}$ have thus a well-defined range $r(\xi)$.
A group~$G$ may act on the graph~$E$ by means of \emph{automorphisms}, \textit{i.e.}, by bijections $E^0 \sqcup E^1 \to E^0 \sqcup E^1$ that preserve the vertex and edge sets as well as the source and the range.
When such a group action is fixed we use the notation $g . v$ and $g . e$ for $g \in G$, $v \in E^0$, $e \in E^1$.

\begin{defn}[{cf.\ \cite[Sec.~2]{ep:ssgautokanca}}]
  A \emph{self-similar graph} $(G, E, \sigma, \phi)$ is given by a countable group~$G$, a finite graph~$E$ without sources, an action~$\sigma$ of~$G$ on~$E$ (denoted by $g . v$ and $g . e)$ and a map $\phi \colon G \times E^1 \to G$ (called \emph{one-cocycle}) with
  \[ \phi(g, e) . v = g . v \,, \qquad \phi(g h, e) = \phi(g, h . e) \phi(h, e) \] for all $g, h \in G$, $v \in E^0$, $e \in E^1$.

  We extend the one-cocycle~$\phi$ on finite paths~$\alpha$ in~$E$ recursively by \[ \phi(g, v) \defeq g \,, \qquad \phi(g, e \alpha) \defeq \phi(\phi(g, e), \alpha) \,, \]
  for $v \in E^0$, $e \in E^1$, and the action~$\sigma$ on finite or infinite paths in~$E$ by letting
  \[ g . (e_1 e_2 e_3 \dots) \defeq (g . e_1) (\phi(g, e_1) . e_2) (\phi(g, e_1 e_2) . e_3) \dots \,. \]
\end{defn}

\begin{exas}\label{exas:self-similar} (1) Every finite graph~$E$ without sources can be seen in a natural way as a self-similar graph with trivial group $G = \{ 1_G \}$.

(2) For \emph{self-similar group actions} we have a single vertex $E^0 = \{ v \}$, and a group~$G$ acts on the edge set~$E^1$.
For instance, the “2-odometer” \cite[Ex.~3.3]{n:goegasa} comprises the free cyclic group $G = \langle a \rangle \cong \Z$ acting on the graph $E = (\{ v \}, \{ e_0, e_1 \})$ by $a . e_0 = e_1$, $a . e_1 = e_0$, and the one-cocycle given by $\phi(a, e_0) = 1_G$, $\phi(a, e_1) = a$.
Another example is the “Grigorchuk group” \cite[Ex.~3.4]{n:goegasa}, in which the free group $G = \langle a, b, c, d \rangle$ interacts with $E = (\{ v \}, \{ e_0, e_1 \})$ as \begin{gather*}
a . e_0 = e_1 \,,\quad a . e_1 = e_0 \,, \qquad \phi(a, e_0) = 1_G \,, \quad \phi(a, e_1) = 1_G \,,\! \\[-1mm]
b . e_0 = e_0 \,,\quad b . e_1 = e_1 \,, \qquad \phi(b, e_0) = a \,, ~\quad\ \phi(b, e_1) = c \,,~ \\[-1mm]
c . e_0 = e_0 \,,\quad c . e_1 = e_1 \,, \qquad \phi(c, e_0) = a \,, ~\quad\ \phi(c, e_1) = d \,,~ \\[-1mm]
d . e_0 = e_0 \,,\quad d . e_1 = e_1 \,, \qquad \phi(d, e_0) = 1_G \,, \quad \phi(d, e_1) = b \,.~ \end{gather*}

(3) A \emph{Katsura algebra} (cf.~\cite[Ex.~3.4]{ep:ssgautokanca}) is a C$^*$-algebra specified by two integer $n \!\times\! n$-matrices $A = (a_{ij})$, $B = (b_{ij})$, where $a_{ij} \ge 0$ and~$A$ has no zero rows, and $a_{ij} = 0$ implies $b_{ij} = 0$.
As in~\cite{ep:ssgautokanca} we may consider an associated self-similar graph, where~$A$ serves as adjacency matrix of a graph~$E$ on~$n$ vertices and edge set $E^1 = \{ e_{ij}^k \mid 0 \le k < a_{ij} \}$.
We consider the free cyclic group $G = \langle g \rangle \cong \Z$ with action and cocycle on the graph~$E$ given by $g . v = v$ for $v \in E^0$ and
\[ g . e_{ij}^k = e_{ij}^r \,, \qquad \phi(g, e_{ij}^k) = g^q \,, \]
where $k + b_{ij} = q a_{ij} + r$ with $0 \le r < a_{ij}$ (division with remainder).

For instance, as in~\cite[Sec.~5.4]{cepss:soaatnhg}, the matrices
\[ A = \begin{pmatrix} 2 & 1 & 0 \\ 1 & 2 & 1 \\ 1 & 1 & 2 \end{pmatrix} \quad \text{ and } \quad  B = \begin{pmatrix} 1 & 2 & 0 \\ 2 & 1 & 2 \\ 0 & 2 & 1 \end{pmatrix} \]
define a graph~$E$ on three nodes with edge set \[ E^1 = \{ e_{11}^0, e_{11}^1, e_{12}, e_{13}, e_{21}, e_{22}^0, e_{22}^1, e_{23}, e_{32}, e_{33}^0, e_{33}^1 \} \,, \]
and the matrix~$B$ determines an action and a one-cocycle by
\begin{gather*}
  g . e_{ii}^0 = e_{ii}^1 \,, \quad \phi(g, e_{ii}^0) = 1_G \text{ for all } i = 1, 2, 3 \,, \\[-1mm]
  g . e_{ii}^1 = e_{ii}^0 \,, \quad \phi(g, e_{ii}^1) = g \text{ for all } i = 1, 2, 3 \,, \\
  g . e_{12} = e_{12} \,, \qquad \phi(g, e_{12}) = g^2 \,, \\[-1mm]
  g . e_{13} = e_{13} \,, \qquad \phi(g, e_{13}) = 1_G \,, \\[-1mm]
  g . e_{21} = e_{21} \,, \qquad \phi(g, e_{21}) = g^2 \,, \\[-1mm]
  g . e_{23} = e_{23} \,, \qquad \phi(g, e_{23}) = g^2 \,, \\[-1mm]
  g . e_{32} = e_{32} \,, \qquad \phi(g, e_{32}) = g^2 \,.
\end{gather*}
Notice that the group $G = \langle g \rangle$ acts like a 2-odometer on paths $\alpha \in E^*$ such that $s(e) = r(e)$ for every edge~$e$ in~$\alpha$, while for paths $\alpha \in E^*$ in which $s(e) \ne r(e)$ for all edges~$e$, one has $\phi(g, \alpha) = 1_G$ if~$e_{13}$ is an edge in~$\alpha$, and $\phi(g, \alpha) = g^{2^{|\alpha|}}$ otherwise. \end{exas}

Now given a self-similar graph $(G, E, \sigma, \phi)$, which we shortly denote as $(G, E)$, we may associate an inverse semigroup as follows (cf.~\cite[Sec.~4]{ep:ssgautokanca}).
Recall that an \emph{inverse semigroup} is a semigroup~$S$ in which for all $s \in S$ there is a unique $s^* \defeq t \in S$ such that $s t s = s$ and $t s t = t$.
Its set of idempotent elements $X(S)$ forms a meet-semilattice under multiplication.  Here we let the semigroup be
\[ S_{G, E} \defeq \{ (\alpha, g, \beta) \mid \alpha, \beta \in E^* ,\, g \in G ,\, s(\alpha) = g . s(\beta) \} \cup \{ 0 \} \,, \]
with multiplication given by
\[ (\alpha, g, \beta) (\gamma, h, \delta) \defeq
  \begin{cases} (\alpha g . \epsilon, \phi(g, \epsilon) h, \delta) &\text{if } \gamma = \beta \epsilon \,, \\
    (\alpha, g \phi(h^{-1}, \epsilon)^{-1}, \delta h^{-1} . \epsilon) & \text{if } \beta = \gamma \epsilon \,, \\
    0 & \text{otherwise} \,. \end{cases} \]
This provides an inverse semigroup where $(\alpha, g, \beta)^* = (\beta, g^{-1}, \alpha)$.  Moreover, we have $X(S_{G, E}) = \{ (\gamma, 1_G, \gamma) \mid \gamma \in E^* \}$, which we may identify with the set~$E^*$ of finite paths.

Next, one considers the action of this semigroup on its so-called tight spectrum~\cite[Sec.~8]{ep:ssgautokanca}.
An \emph{action} of an inverse semigroup~$S$ on a space~$X$ is given by a collection $\sigma = \{ \sigma_s \}_{s \in S}$ of homeomorphisms $\sigma_s \colon D(s^* s) \longrightarrow D(s s^*)$ between open subsets of~$X$ such that $\sigma_s \circ \sigma_t = \sigma_{s t}$ for all $s, t \in S$ and $\bigcup_{s \in S} D(s^* s) = X$.
In the present situation the space (called \emph{tight spectrum}) on which the semigroup $S \defeq S_{G, E}$ acts can be identified with the compact space $X \defeq E^{\infty}$ of infinite paths in~$E$, in which the cylinder sets \[ Z(\gamma) \defeq \{ \gamma \xi \mid \xi \in E^{\infty} ,\, r(\xi) = s(\gamma) \} \] serve as a basis for its topology.
Given $s = (\alpha, g, \beta) \in S_{G, E}$ we have $D(s^* s) = Z(\beta)$, $D(s s^*) = Z(\alpha)$, and the homeomorphism of the action is given by
\[ \sigma_s \colon Z(\beta) \longrightarrow Z(\alpha) \,, \qquad (\alpha, g, \beta) . \beta \xi \defeq \alpha g . \xi \,. \]

Finally, given such an action one may define the \emph{groupoid of germs} as in~\cite[Sec.~4]{e:isacca} and~\cite[Sec.~8]{ep:ssgautokanca}), which is termed the \emph{tight groupoid} of the inverse semigroup, cf.~\cite{e:isacca,ep:ttgoais}.  For a self-similar graph $(G, E)$ this groupoid is
\[ \mcG_{G, E} \defeq \{ [\alpha, g, \beta; \eta] \mid \eta = \beta \xi \} \,, \]
where $[s; \eta] = [t; \mu]$ if and only\! if $\eta = \mu$ and there is a nonzero idempotent $u \in S_{G, E}$ with $u . \eta = \eta$ and $s u = t u$.
The unit space $\smash{\mcGo_{G, E}} = \{ [\gamma, 1_G, \gamma; \eta] \mid \eta = \gamma \xi \}$ may be identified with the infinite paths~$E^{\infty}$, so that the source and range maps are
\[ s([\alpha, g, \beta; \beta \xi]) = \beta \xi \quad\text{ and }\quad r([\alpha, g, \beta; \beta \xi]) = \alpha g . \xi \,. \]
A topology basis on~$\mcG_{G, E}$ is given by compact open bisections of the form
\[ \Theta(\alpha, g, \beta; Z(\gamma)) \defeq \{ [\alpha, g, \beta; \eta] \in \mcG_{G, E} \mid \eta \in Z(\gamma) \} \]
for $\alpha, \beta, \gamma \in E^*$ and $g \in G$.  Thus $\mcG_{G, E}$ is a locally compact, second-countable and ample groupoid.
In~\cite[Secs.~12--14]{ep:ssgautokanca}, based on results in~\cite{ep:ttgoais}, characterizations are given for~$\mcG_{G, E}$ to be Hausdorff, minimal and effective, in terms of properties of the self-similar graph $(G, E, \sigma, \phi)$ and the action of $S_{G, E}$ on $E^{\infty}$.
In particular, by~\cite[Th.~12.2]{ep:ssgautokanca}, the groupoid~$\mcG_{G, E}$ is Hausdorff if and only if for any $g \in G$ there are only finitely many minimal strongly fixed paths, where a \emph{strongly fixed path} is a path $\alpha \in E^*$ such that $g . \alpha = \alpha$ and $\phi(g, \alpha) = 1_G$.

The following condition helps in proving that any compact open subset of~$\mcG_{G, E}$ is regular open.
For $s = (\alpha, g, \beta) \in S_{G, E}$ we denote its set of \emph{fixed elements} by
\[ F_s \defeq \{ \xi \in E^{\infty} \mid (\alpha, g, \beta) . \xi = \xi \} \,, \]
and its set of \emph{trivially fixed elements} by
\[ T\!F_s \defeq \{ \xi \in E^{\infty} \mid \exists u = (\gamma, 1_G, \gamma) \in S_{G, E} \text{ with } s u = u \text{ and } \xi \in Z(\gamma) \} \,. \]
Following \cite[Def.~5.4]{cepss:soaatnhg}, the inverse semigroup~$S_{G, E}$ is said to satisfy \emph{Condition}\,(S) if for any finite set $\{ s_1, \dots, s_n \} \subseteq S_{G, E}$ of non-idempotents,
\[ \xi \in \bigcap_{i=1}^n (F_{s_i} \!\setminus\! T\!F_{s_i}) \quad\text{ implies }\quad \xi \notin (\bigcup_{i=1}^n F_{s_i})^{\circ} . \]
We then have the following.

\begin{thm}[{cf.~\cite[Th.~5.11\,(3)]{cepss:soaatnhg}}]\label{sim-ssgalg}
Let $(G, E, \sigma, \phi)$ be a self-similar graph such that $S_{G, E}$ satisfies Condition\,{\upshape (S)} and such that $\mcG_{G, E}$ is minimal, and let~$K$ be a semifield.  Then the algebra $A_K(\mcG_{G, E})$ is congruence-simple if and only if~$K$ is either a field or the Boolean semifield.
\end{thm}	

\begin{proof} Applying Condition\,(S) for $n = 1$ yields $F_s^{\circ} \subseteq T\!F_s$ for all $s \in S$, whence $\mcG_{G, E}$ is topologically free as in \cite[Def.~4.1]{ep:ttgoais} and thus effective by \cite[Th.~4.7]{ep:ttgoais}.
Now by \cite[Lem.~5.6]{cepss:soaatnhg}, Condition\,(S) implies that any compact open subset of $\mcG_{G, E}$ is regular open.
Then, by \cite[Th.~5.11\,(3)]{cepss:soaatnhg} and Corollary~\ref{sim-cor}, we immediately obtain the result. \end{proof}	

We remark that \cite[Lem.~5.9]{cepss:soaatnhg} gives a criterion for verifying Condition\,(S) more readily in this context, and the minimality assumption in Theorem~\ref{sim-ssgalg} is satisfied in many cases; see \cite[Th.~13.6]{ep:ssgautokanca}.
By using Theorem~\ref{sim-ssgalg} and Corollary~\ref{sim-cor}, we are able to provide congruence-simple Steinberg algebras of certain non-Hausdorff ample groupoids over the Boolean semifield~$\B$.

Consider first the self-similar graph action of the Katsura algebra specified by the $3 \!\times\! 3$-matrices of Example~\ref{exas:self-similar}\,(3).  Denote by $\mcG_{\Z, E_A}$ and $S_{\Z, E_A}$ its associated groupoid and inverse semigroup, respectively.
Then~$\mcG_{\Z, E_A}$ is minimal and non-Hausdorff, by \cite[Lem.~5.12]{cepss:soaatnhg}, and~$S_{\Z, E_A}$ satisfies Condition\,(S), by \cite[Lem.~5.15]{cepss:soaatnhg}.  Moreover, $A_K(\mcG_{\Z, E_A})$ is simple for all fields~$K$ by \cite[Th.~5.16 (2)]{cepss:soaatnhg}.
Using these observations, Corollary~\ref{sim-cor}, and Theorem~\ref{sim-ssgalg}, we immediately obtain the following result.

\begin{cor}[{cf.~\cite[Th.~5.16 (2)]{cepss:soaatnhg}}]\label{sim-Katalg}
  For any semifield~$K$, the algebra $A_K(\mcG_{\Z, E_A})$ is congruence-simple if and only if~$K$ is either a field or the Boolean semifield~$\B$.
\end{cor}

Next consider a self-similar group action $(G, X)$ as in Example~\ref{exas:self-similar}\,(2), in which a group~$G$ acts on the edge set $X \defeq E^1$ of a one-vertex graph.
Then its associated groupoid, which we denote by~$\mcG_{G, X}$, is minimal and effective; see \cite[Sec.~17]{ep:ssgautokanca}.

Following \cite[Def.~5.19]{cepss:soaatnhg}, the self-similar action is called \emph{$\omega$-faithful} if whenever $\xi \in X^{\infty}$ and $g_1, \dots, g_n \in G$ satisfy $g_i . \gamma = \gamma$ and $\phi(g_i, \gamma) \ne 1_G$ for all~$i$ and all finite prefixes~$\gamma$ of~$\xi$,
then there is some $m \in \mathbb N$ such that for every finite prefix~$\gamma$ of~$\xi$ of length $|\gamma| \ge m$ there exists $\alpha \in X^*$ with $\phi(g_i, \gamma) . \alpha \ne \alpha$.
In \cite[Lem.~5.20]{cepss:soaatnhg}, the authors showed that $S_{G, X}$ satisfies Condition\,(S) for all $\omega$-faithful self-similar actions.  Moreover, $A_K(\mcG_{G, X})$ is simple for all $\omega$-faithful self-similar actions and for all fields~$K$, by \cite[Th.~5.21\,(3)]{cepss:soaatnhg}.
By employing these results, Corollary~\ref{sim-cor}, and Theorem~\ref{sim-ssgalg}, we immediately obtain the following.

\begin{cor}[{cf.~\cite[Th.~5.21\,(3)]{cepss:soaatnhg}}]\label{sim-ssacalg}
  For any semifield~$K$ and any $\omega$-faithful self-similar action $(G, X)$, the algebra $A_K(\mcG_{G, X})$ is congruence-simple if and only if~$K$ is either a field or the Boolean semifield~$\B$.
\end{cor}

For instance, one easily sees that the 2-odometer of Example~\ref{exas:self-similar}\,(2) is $\omega$-faithful, since if $g . \gamma = \gamma$ for all finite prefixes~$\gamma$ of some $\xi \in X^{\infty}$ we must have $g = 1_G$ and thus $\phi(g, \gamma) = 1_G$, for any $g \in G = \langle a \rangle$.

On the other hand, the Grigorchuk group $G = \langle a, b, c, d \rangle$ of Example~\ref{exas:self-similar}\,(2) is not $\omega$-faithful, which can be checked by considering $\xi = 1 1 1 \ldots \in X^{\infty}$ and $b, c, d \in G$.
In fact, it has been shown that the Steinberg algebra of the tight groupoid associated to the Grigorchuk group is simple over a ground field of characteristic zero (see \cite[Th.~5.22]{cepss:soaatnhg}), but not over the field~$\mathbb F_2$ (cf.\ \cite[Ex.~4.5]{n:goegasa} and \cite[Cor.~5.26]{cepss:soaatnhg}).
In light of this interesting observation we conclude the article with the following question.

\begin{op} Is the Steinberg algebra of the tight groupoid associated to the Grigorchuk group over the Boolean semifield congruence-simple? \end{op}

\end{document}